\documentclass{article}

\usepackage{geometry}
\usepackage{amssymb}
\usepackage{amsmath}

\newtheorem{theorem}{Theorem}[section]
\newtheorem{lemma}[theorem]{Lemma}
\newtheorem{corollary}[theorem]{Corollary}
\newtheorem{proposition}[theorem]{Proposition}
\newenvironment{proof}{{\bf Proof.}}{\hfill$\square$\vskip 2ex}

\newcommand{\R}{\mathbb{R}}
\newcommand{\N}{\mathbb{N}}

\newcommand{\const}{\mathrm{const}}

\newcommand{\norm}[1]{\left\lVert #1 \right\rVert}
\newcommand{\LL}{\mathcal{L}}
\newcommand{\sgn}{\mathrm{sign}}
\newcommand{\Vect}{\mathrm{Vect}}

\title{Dissipative prolongations of the multipeakon solutions to the Camassa-Holm equation}
\author{Wojciech Kry\'nski\thanks{
Institute of Mathematics, Polish Academy of Sciences, ul.~\'Sniadeckich 8, 00-656 Warszawa, Poland
\newline 
E-mail: krynski@impan.pl
}}

\begin{document}
\maketitle
\begin{abstract}
Multipeakons are special solutions to the Camassa-Holm equation described by an integrable geodesic flow on a Riemannian manifold. We present a bi-Hamiltonian formulation of the system explicitly and write down formulae for the associated first integrals. Then we exploit the first integrals and present a novel approach to the problem of the dissipative prolongations of multipeakons after the collision time. We prove that an $n$-peakon after a collision becomes an $n-1$-peakon for which the momentum is preserved.
\end{abstract}

\section{Introduction}
We shall consider the Camassa-Holm equation \cite{CH} in the following form
\begin{equation}\label{eqCH}
u_t-u_{xxt}+3uu_x-2u_xu_{xx}-uu_{xxx}=0.
\end{equation}
An $n$-peakon solution $u(t,x)$, where $t\geq 0$ and $x\in\R$, to equation \eqref{eqCH} is given by the formula
$$
u(t,x)=p_1(t)e^{-|x-q_1(t)|}+\cdots+p_n(t)e^{-|x-q_n(t)|}
$$
where $(q,p)=(q_i,p_i)_{i=1,\ldots,n}$ solves the Hamiltonian system
\begin{equation}\label{hameq}
\dot q_i=\frac{\partial H}{\partial p_i},\qquad \dot p_i=-\frac{\partial H}{\partial q_i},\qquad i=1,\ldots,n,
\end{equation}
with
$$
H(q,p)=\frac{1}{2}\sum_{i,j=1}^np_ip_je^{-|q_i-q_j|}.
$$
Denote 
$$
h=(e^{-|q_i-q_j|})_{i,j=1,\ldots,n}
$$
and
$$
g=h^{-1}.
$$
It follows that the $n$-peakon solutions to the Camassa-Holm equation are in a one to one correspondence with geodesics $t\mapsto q(t)$, $t\geq 0$, on $n$-dimensional manifold with coordinates $q=(q_1,\ldots,q_n)$ and metric $g$. The metric is well defined and has Riemannian signature unless $q_i=q_j$ for some $i\neq j$. If a geodesic approaches the singular set $\{q_i=q_j\}$ as $t\to t^*\in\R$ then we say that  two peakons collide. The corresponding time $t^*$  is called the collision time. In this paper we address the problem of a prolongation of an $n$-peakon solution after a collision. We consider dissipative weak solutions.

\paragraph{Dissipative weak solutions.} A weak solution $u=u(t,x)$ to the Camassa-Holm equation is called dissipative \cite{BC} if the following two conditions are satisfied
\begin{align}
&\partial_xu(t,x)\leq \mathrm{const}\left(1+\frac{1}{t}\right),\label{dis1}\\ 
&\norm{u(t,.)}_{H^1(\R)}\leq\norm{u(0,.)}_{H^1(\R)},\label{dis2}
\end{align}
for all $t>0$. It is known \cite{BC,J} that the dissipative weak solutions are unique.
Note that \eqref{dis1} is the Oleinik-type condition and \eqref{dis2} can be thought of as a weak-energy condition. It can be easily checked that both \eqref{dis1} and \eqref{dis2} are satisfied for $n$-peakons until the collision time. Indeed, on the one hand, $\partial_xu(t,x)$ is decreasing in time, and, on the other hand,
$$
\norm{u(t,.)}_{H^1(\R)}=H(q(t),p(t)),
$$
i.e. $\norm{u(t,.)}_{H^1(\R)}$ is constant. 

\paragraph{Overview.}
The Camassa-Holm equation is integrable \cite{CH,CHH,C1}. This manifests, for instance, in a fact that the inverse-scattering method can be applied for \eqref{eqCH} (see \cite{BSS, BSS2,C2} and also \cite{J,HR,LZ}) or that \eqref{eqCH} has a bi-Hamiltonian formulation \cite{KM}. Our aim in the next section is to exploit the integrable nature of the Camassa-Holm equation and write down a bi-Hamiltonian system explicitly for \eqref{hameq}. For this we need an additional Hamiltonian function alongside with a symplectic structure that is compatible with the standard one underlying \eqref{hameq}. The additional Hamiltonian function is constructed as the conserved quantity, which will be referred to as the momentum, that is associated to a Killing vector field of metric $g$. Having the bi-Hamiltonian structure we get a sequence of first integrals for \eqref{hameq} which appear to be polynomial in $p$. We present explicit formulae for the first integrals in Theorem \ref{thm_first_int}.

The  bi-Hamiltonian formulation is utilized in Sections 3 and 4 to analyse the behaviour of the solutions to \eqref{hameq} near the collision points. Section 3 is devoted to the simpler case of 2-peakons, while in Section 4 we treat the general case of $n$-peakons. The problem of dissipative prolongation of 2-peakons has been solved recently in \cite{GH}. We present an alternative approach in this case that is based on geometric properties of the first integrals (Theorem \ref{thm2}). In higher dimensions we extensively use the bi-Hamiltonian structure which allow us to deal with the solutions at singular points (Lemma \ref{lemma4}) and consequently prove the main Theorem \ref{thm3}. This is a novel approach to the problem of the dissipative prolongation of multipeakons previously investigated in \cite{BC,HR}. We believe that our method is of independent interest and can be further exploited in other works.

The key ingredient of the proof relies on the continuity of roots of polynomials \cite{W}. Namely, this fact is used to prove that a geodesic $t\mapsto q(t)$ behaves regularly in a neighbourhood of the singular set. We exploit here an observation that the first integrals we get from the bi-Hamiltonian approach are polynomial in $p$. However, an additional reasoning is needed to overcome a difficulty that comes form the fact that $p$ explodes in a neighbourhood of the singular set.

We refer to recent papers \cite{Ciin,EG,GH,J} for more information on the current progress in the theory of the Camassa-Holm equation. In particular \cite{Ciin} is devoted to geometric properties of $g$. 

\paragraph{Acknowledges.} I am very thankful to Tomasz Cie\'slak for presenting me the problem considered in this paper and for inspiring discussions.

\section{Momentum and integrability}
We start with the following observation.
\begin{proposition}\label{prop1}
The vector field $X=\sum_{i=1}^n\partial_{q_i}$ is a Killing vector field for the metric $g$.
\end{proposition}
Recall that a vector field $X$ is a Killing vector for $g$ if its flow preserves the metric. Equivalently
$$
\LL_Xg=0,
$$
where $\LL_X$ stands for the Lie derivative associated to $X$. This can be rewritten as
\begin{equation}\label{killing}
Xg(Y,Z)=g([X,Y],Z)+g(Y,[X,Z]).
\end{equation}
which should hold for arbitrary vector fields $Y,Z$, where $[.,.]$ is the Lie bracket of vector fields. We get the following

\begin{proposition}\label{prop2}
The function
$$
H_0(q,p)=p(X)=\sum_{i=1}^np_i
$$
is a first integral of \eqref{hameq}. Along a geodesic $t\mapsto g(t)$ the following holds $H_0(q(t),p(t))=g(\dot q(t),X)$.
\end{proposition}
\begin{proof}
Let $t\mapsto q(t)$ be a geodesic of $g$. The corresponding solution to \eqref{hameq} has the form $t\mapsto(q(t),p(t))$ with $p(t)=g(\dot q(t),.)$. Thus along the geodesic $H_0(q(t),p(t))=g(\dot q(t),X)$. The corollary follows from the fact that the scalar product $g(\dot q(t),X)$ is a conserved quantity along any geodesics for any Killing vector field $X$. Indeed, denoting by $\nabla$ the Levi-Civita connection of $g$,  we have
$$
\frac{d}{dt}g(\dot q,X)=g(\nabla_{\dot q}\dot q,X)+g(\dot q,\nabla_{\dot q}X)=g(\dot q,\nabla_X\dot q)-g(\dot q,[X,\dot q])=\frac{1}{2}Xg(\dot q,\dot q)-\frac{1}{2}Xg(\dot q,\dot q)=0.
$$
where we used \eqref{killing} and the metricity and torsion-freeness of $\nabla$.
\end{proof}
Leter on we shall construct a sequence $(H_i)_{i=0,\ldots,n-1}$ of first integrals. The Hamiltonian function $H$ will appear in the sequence as $H_1$, i.e.
$$
H_1(q,p)=\frac{1}{2}g(p,p)=\frac{1}{2}\sum_{i,j=1}^np_ip_je^{-|q_i-q_j|}.
$$
Note that
$$
H_1(q,p)=\frac{1}{2}h(p,p),
$$
where the bi-linear form $h$ (the inverse of $g$) is well defined for all $q$ even if $q_i=q_j$. We shall refer to $H_0$ and $H_1$ as the momentum and the energy, respectively.

\paragraph{Bi-Hamiltonian system.} 
Let $P$ be a bi-vector on $m$-dimensional manifold $M$. By definition it is a skew-symmetric $(2,0)$-tensor. In local coordinates $(x_1,\ldots,x_m)$ it is an expression of the form
$$
P=\sum_{i,j=1}^mP^{ij}\partial_{x_i}\wedge\partial_{x_j}
$$
where $\wedge$ is the exterior tensor product. A bi-vector defines a bracket on the set of $C^\infty$ functions on $M$ by the formula
$$
\{\phi,\psi\}_P=P(d\phi,d\psi)=\frac{1}{2}\sum_{i,j=1}^mP^{ij}((\partial_{x_i}\phi)(\partial_{x_j}\psi)-(\partial_{x_i}\psi)(\partial_{x_j}\phi)),
$$
where $d\phi$ and $d\psi$ are exterior differentials of functions $\phi$ and $\psi$. $P$ is called a Poisson structure if $\{.,.\}_P$ satisfies the Jacobi identity.

Note that any bi-vector can be considered as a mapping $\Omega^1(M)\to \Vect(M)$ that sends 1-forms to vector fields on $M$. Indeed, for any 1-form $\alpha=\sum_{i=1}^m\alpha_idx_i$
$$
\vec \alpha_P=P(\alpha,.)=\frac{1}{2}\sum_{i,j=1}^mP^{ij}(\alpha_i\partial_{x_j}-\alpha_j\partial_{x_i})
$$
is a vector field on $M$. In particular, for a given function $H$ a Poisson structure defines the Hamiltonian vector field by formula
$$
\vec H_P=P(dH,.).
$$
If $m=2n$ is even and $P$ is non-degenerate, meaning that $P$ is invertible as a map $\Omega^1(M)\to \Vect(M)$, then one can choose local coordinates on $M$ such that $P=\sum_{i=1}^n\partial_{p_i}\wedge\partial_{q_i}$. Then the integral curves of $\vec H_P$ satisfy the standard Hamiltonian equation \eqref{hameq}.

A bi-Hamiltonian structure on $M$ is a pair of compatible Poisson structures $(P_0,P_1)$, in a sense that any linear combination $\lambda_0 P_0+\lambda_1 P_1$ is a Poisson structure too. We shall assume that $P_1$ is non-degenerate.

The bi-Hamiltonian approach to the integrability goes back to \cite{M} and \cite{DG} and relies on the assumption that there are given two functions $H_0$ and $H_1$ such that
$$
P_0(dH_0,.)=P_1(dH_1,.)
$$
i.e. $\vec{H_0}_{P_0}=\vec{H_1}_{P_1}$. Then we say that $\dot x =\vec{H_0}_{P_0}(x)=\vec{H_1}_{P_1}(x)$ is a bi-Hamiltonian system. The bi-Hamiltonian systems are integrable. Namely, one considers $P_0$ and $P_1$ as mappings $\Omega^1(M)\to \Vect(M)$ and takes the composition
$$
S=P_1^{-1}\circ P_0\colon \Omega^1(M)\to \Omega^1(M)
$$
called the recursive operator. Then, for any $i\in\N$, the one-form
\begin{equation}\label{firstint}
S^i(dH_1)
\end{equation}
is closed and equals $dH_{i+1}$ for some function $H_{i+1}$ \cite[Theorem 3.4]{DG} which is proved to be a first integral of $\vec{H_0}_{P_0}$ (and consequently of $\vec{H_1}_{P_1}$). Moreover, the functions $H_i$ are in involution. 

System \eqref{hameq} for the Camassa-Holm equation has the following bi-Hamiltonian formulation.
\begin{theorem}\label{thmbiham}
The pair
$$
\begin{aligned}
P_0&=\sum_{i,j=1}^np_ie^{-|q_i-q_j|}\partial_{p_i}\wedge\partial_{q_j} -\sum_{i<j}\sgn(q_i-q_j)p_ip_je^{-|q_i-q_j|}\partial_{p_i}\wedge\partial_{p_j}\\
&+\sum_{i<j}\sgn(q_i-q_j)(e^{-|q_i-q_j|}-1)\partial_{q_i}\wedge\partial_{q_j}\\
P_1&=\sum_{i=1}^n\partial_{p_i}\wedge\partial_{q_i}
\end{aligned}
$$
is a bi-Hamiltonian structure such that $P_0(dH_0,.)=P_1(dH_1,.)$ where $H_0$ is the momentum and $H_1$ is the energy for the $n$-peakon solutions of the Camassa-Holm equation.
\end{theorem}
\begin{proof}
One immediately checks that
$$
P_0(dH_0,.)=P_1(dH_1,.)=\sum_{i,j=1}^mp_ie^{-|q_i-q_j|}\partial_{q_j}-\sum_{i<j}\sgn(q_i-q_j) p_ip_je^{-|q_i-q_j|}(\partial_{p_j}-\partial_{p_i}).
$$
Thus it is enough to prove that $P_0$ and $P_1$ are compatible Poisson brackets. To do this it is enough to verify that the Schouten-Nijenhuis brackets $[P_0,P_0]_{SN}$, $[P_1,P_1]_{SN}$ and $[P_0,P_1]_{SN}$ vanish. The Schouten-Nijenhuis bracket of two bi-vectors $P$ and $R$, defined in a coordinate system $(x_1,\ldots,x_m)$ by skew-symmetric coefficients $P^{ij}$ and $R^{ij}$, respectively, is given by the formula
$$
[P,R]_{SN}=\sum_{i,j,k,l=1}^m(P^{ij}\partial_{x_i}R^{kl}+R^{ij}\partial_{x_i}P^{kl})
\partial_{x_j}\wedge\partial_{x_k}\wedge\partial_{x_l}.
$$
Hence, $[P,R]_{SN}=0$ if and only if
$$
\sum_{i=1}^n(P^{ij}\partial_{x_i}R^{kl}+P^{ik}\partial_{x_i}R^{lj}+P^{il}\partial_{x_i}R^{jk} +R^{ij}\partial_{x_i}P^{kl}+R^{ik}\partial_{x_i}P^{lj}+R^{il}\partial_{x_i}P^{jk})=0
$$
for all $j<k<l$. In the setting of the theorem clearly $[P_1,P_1]_{SN}=0$ holds. It is also easy to prove that $[P_0,P_1]_{SN}=0$ because the coefficients of $P_1$ are constant and one considers the cyclic sums of derivatives of the coefficients of $P_0$. The most complicated identity $[P_0,P_0]_{SN}=0$ can be verified by direct computations. For instance, considering the coefficient of $[P_0,P_0]_{SN}$ next to $\partial_{q_j}\wedge\partial_{q_k}\wedge\partial_{q_l}$ with $q_j>q_k>q_l$ one gets
$$
\begin{aligned}
&(e^{-q_j+q_k}-1)e^{-q_k+q_l}-(e^{-q_j+q_l}-1)e^{-q_k+q_l}+
(e^{-q_k+q_l}-1)e^{-q_j+q_l}+(e^{-q_j+q_k}-1)e^{-q_j+q_l}+\\
&-(e^{-q_j+q_l}-1)e^{-q_j+q_k}+(e^{-q_k+q_l}-1)e^{-q_j+q_k}=\\
&e^{-q_j+q_l}-e^{-q_k+q_l}-e^{-q_j-q_k+2q_l}+e^{-q_k+q_l}+
e^{-q_j-q_k+2q_l}-e^{-q_j+q_l}+e^{-2q_j+q_k+q_l}-e^{-q_j+q_l}+\\
&-e^{-2q_j+q_k+q_l}+e^{-q_j+q_k}+e^{-q_j+q_l}-e^{-q_j+q_k}=0.
\end{aligned}
$$
The other coefficients can be computed similarly. We omit the calculations here.
\end{proof}

\begin{corollary}
The first integrals $H_i$ defined by the recursive operator via \eqref{firstint} for the bi-Hamiltonian system of Theorem \ref{thmbiham} are polynomial functions of $p$. Precisely $H_i$ is a polynomial of degree $i+1$.
\end{corollary}
\begin{proof}
Note that $P_0$ is polynomial in $p$ (of degree 2 next to $\partial_{p_i}\wedge\partial_{p_j}$, of degree 1 next to $\partial_{p_i}\wedge\partial_{q_j}$ and of degree 0 next to $\partial_{q_i}\wedge\partial_{q_j}$). Additionally $P^{-1}_1(\partial_{q_i})=dp_i$  and $P^{-1}_1(\partial_{p_i})=-dq_i$. Thus, applying $S=P^{-1}_1\circ P_0$ to a one-form polynomial of degree $k$ in $p$ one gets  a one-form which is polynomial of degree $k+1$ in $p$. The corollary follows by induction from the fact that $H_0$ is linear in $p$.
\end{proof}

We get that the Hamiltonian system \eqref{hameq} is integrable. Below, in Theorem \ref{thm_first_int}, we shall present formulae for the first integrals. However, we shall introduce some notation first. Namely, let $\mathcal{I}_s^n$ be the set of all multi-indices $I=(i_0,\ldots,i_s)$ such that $i_j\in\{1,\ldots,n\}$, for any $j=0,\ldots,s$, and satisfying $i_0\leq i_1\leq\ldots\leq i_s$.  For a given multi-index $I=(i_0,\ldots,i_s)\in\mathcal{I}_s^n$  let $t$ be the number of different values taken by entries of $I$. We define numbers $s^I_0,s^I_1,\ldots,s^I_t$ as multiplicities of the values, i.e. we have
$$
i_0=i_1=\cdots=i_{s^I_0-1}<i_{s^I_0}=\cdots=i_{s^I_0+s^I_1-1}<\cdots< i_{s^I_0+\cdots+s^I_{t-1}}=\cdots=i_{s^I_0+\cdots+s^I_t-1}
$$
and $s^I_0+\cdots+s^I_t=s+1$. We will write
$$
I!=s_0^I!s_1^I!\cdots s_t^I!.
$$
Note that $\frac{(s+1)!}{I!}$ is the number of different sequences that are permutations of the sequence $(i_0,\ldots,i_s)$.

For a given multi-index $I=(i_0,\ldots,i_s)$ we shall write
$$
p_I=p_{i_0}p_{i_1}\cdots p_{i_s}.
$$
Further, let $\mathcal{P}_s$ be the set of all mappings $\rho\colon\{1,\ldots,s\}\to\{0,\ldots,s-1\}$ such that
$$
\rho(j)<j.
$$
For a given $\rho\in\mathcal{P}_s$ and $j\in\{0,\ldots,s-1\}$ we shall denote by $|\rho^{-1}(j)|$ the size of the preimage of $j$, i.e. the number of different $i\in\{1,\ldots,s\}$ such that $\rho(i)=j$. Then
$$
s_\rho=|\{j\ :\ |\rho^{-1}(j)|=1\}|
$$
is the size of the maximal subset $A\subset\{1,\ldots,s\}$ such that $\rho|_A$ is invertible.

\begin{theorem}\label{thm_first_int}
The first integrals of \eqref{hameq} constructed via \eqref{firstint} are given in the region $q_1>\ldots>q_n$ by the following formula
$$
H_s(q,p)=\sum_{I\in\mathcal{I}_s^n}\frac{1}{I!}p_I\left(\sum_{\rho\in\mathcal{P}_s}c_\rho e^{-\sum_{j=0}^s(q_{i_j}-q_{i_{\rho(j)}})}\right)
$$
where
$$
c_\rho=
\left\{
	\begin{array}{ll}
		0  & \mbox{if\ \ } |\rho^{-1}(0)|>1 \mbox{\ \  or\ \ } |\rho^{-1}(j)|>2 \mbox{\ \ for some\ \ } j\in\{1,\ldots,s-1\} \\
		2^{s_\rho} & \mbox{otherwise.}
	\end{array}
\right.
$$
Moreover
\begin{equation}\label{suma}
\sum_{\rho\in\mathcal{P}_s}c_\rho=s!.
\end{equation}
\end{theorem}
\begin{proof}
Firstly we shall prove \eqref{suma} by induction. If $s=1$ then $\mathcal{P}_s$ consists of one mapping: $\rho(1)=0$ with $s_\rho=1$. Thus \eqref{suma} holds in this case.  To prove the formula for $s+1$ we take $\rho\in\mathcal{P}_{s+1}$ and consider its restriction to the set $\{1,\ldots,s\}$, denoted $\tilde\rho$. Hence $\tilde \rho\in\mathcal{P}_s$. Assume that $c_{\tilde\rho}\neq 0$ (otherwise $c_\rho=0$ too, and consequently $\rho$ does not contribute to \eqref{suma}). The crucial observation is that there is exactly the same number of $j\in\{1,\ldots,s-1\}$ such that $|\tilde \rho^{-1}(j)|=0$ and such that $|\tilde \rho^{-1}(j)|=2$. Denote $A_0=\{j\in\{1,\ldots,s-1\}\ :\ |\tilde \rho^{-1}(j)|=0\}$, $A_1=\{j\in\{1,\ldots,s-1\}\ :\ |\tilde \rho^{-1}(j)|=1\}$ and $A_2=\{j\in\{1,\ldots,s-1\}\ :\ |\tilde \rho^{-1}(j)|=2\}$. Now, if $\rho(s+1)\in A_0$ then $c_\rho=2c_{\tilde\rho}$, if $\rho(s+1)\in A_1$ then $c_\rho=c_{\tilde\rho}$ and if $\rho(s+1)\in A_2$ then $c_\rho=0$. Moreover, if $\rho(s+1)=0$ then $c_\rho=0$ and if $\rho(s+1)=s$ then $c_\rho=2$. Hence, since $|A_0|=|A_2|$ and the sets $A_0$, $A_1$ and $A_2$ together with $\{0,s\}$ cover all $\{0,\ldots,s\}$, we get that for any fixed $\tilde\rho\in\mathcal{P}_s$ the following holds
$$
\sum_{\substack{\rho\in\mathcal{P}_{s+1},\\ \rho|_{\{1,\ldots,s\}}=\tilde\rho}}c_\rho=(s+1)c_{\tilde\rho},
$$
and this implies \eqref{suma} for $s+1$.

Now, we proceed to the proof of the formula for $H_s$. We shall proceed by induction again. For $s=0$ the formula is correct. In order to find $H_{s+1}$ we apply the recursive operator $S$ to $dH_s$. We already know that $H_{s+1}$ is a homogeneous polynomial of order $s+2$ in $p_i$'s. Our goal is to compute the coefficients of the polynomial which are exponential functions of $q_i$'s.  Since the system is bi-Hamiltonian we know that $S(dH_s)$ is a closed one-form. Therefore, in order to integrate it is enough to consider the terms involving $dq_i$'s and $dp_i$'s separately (integration of both terms coincide).

We start with  the coefficient of $H_{s+1}$ next to $p_i^{s+2}$, for some fixed $i$. According to \eqref{suma} it should equal $\frac{1}{s+2}$ because $I!=(s+2)!$ in this case. We verify this directly by looking for the coefficient of $S(dH_s)$ next to $p_i^{s+1}dp_i$, as it is the only possibility to get a term involving $p_i^{s+2}$ in  $H_{s+1}$. One sees that, under the inductive hypothesis, it equals $(s+1)\frac{1}{(s+1)}=1$. Therefore, after integration, we get $\frac{1}{s+2}$ next to $p_i^{s+2}$ in $H_{s+1}$.

In order to compute other coefficients it will be more convenient to look for the terms next to $dq_i$'s in $S(dH_s)$. We have
\begin{equation}\label{dhs}
\begin{aligned}
S(dH_s)=&\sum_{I\in\mathcal{I}_s^n}\frac{1}{I!}\left(\sum_{\rho\in\mathcal{P}_s}\sum_{j=0}^s\sum_{k=1}^n c_\rho e^{-|q_{i_j}-q_k|-\sum_{l=0}^s(q_{i_l}-q_{i_{\rho(l)}})}\sgn(q_k-q_{i_j})p_Ip_k dq_k\right)\\
&+\sum_{I\in\mathcal{I}_s^n}\frac{1}{I!}\left(\sum_{\rho\in\mathcal{P}_s}\sum_{j=0}^s\sum_{k=1}^n c_\rho e^{-|q_{i_j}-q_k|-\sum_{l=0}^s(q_{i_l}-q_{i_{\rho(l)}})}\sgn(\rho,j)p_Ip_k dq_k\right) \mod dp.
\end{aligned}
\end{equation}
where
$$
\sgn(\rho,j)=
\left\{
	\begin{array}{ll}
		-1  & \mbox{if\ \ } j=0 \mbox{\ \  or\ \ } |\rho^{-1}(j)|=2\\
		1 & \mbox{if\ \ } j=s \mbox{\ \ or\ \ } |\rho^{-1}(j)|=0\\
		0 & \mbox{otherwise.}
	\end{array}
\right.
$$
The first sum in \eqref{dhs} comes from $S$ applied to the part of $dH_s$ involving $dp_i$'s and the second sum in \eqref{dhs} comes from $S$ applied to the part of $dH_s$ involving $dq_i$'s. Let us notice that it is enough to consider only those values of $k$ in the two sums above for which $k\geq i_j$ for all $j=0,\ldots,s$, where $I=(i_0,\ldots,i_s)\in\mathcal{I}_s^n$ is a multi-index. Indeed, in a process of integration of these particular components we will get all coefficients in $H_{s+1}$ that we are looking for, and we \emph{a priori} know that integration of the remaining parts is consistent with the integration of these particular coefficients because, again, $S(dH_s)$ is a closed one-form.

So, let us fix $I\in\mathcal{I}_s^n$, $k\in\{1,\ldots,n\}$ such that $k\geq i_j$ for all $j=0,\ldots,s$, as well as $\rho\in\mathcal{P}_s$ and $j\in\{0,\ldots,s\}$. We extend $I$ to a multi-index $\hat I\in\mathcal{I}_{s+1}^n$ by setting $i_{s+1}=k$ and similarly we extend $\rho$ to a mapping $\hat \rho\in\mathcal{P}_{s+1}$ such that $\hat\rho(s+1)=j$. Our aim is to find the coefficient $c_{\hat\rho}$ in $H_{s+1}$. Note that the components of the first and the second sum in \eqref{dhs} corresponding to $I$, $\rho$, $j$ and $k$ cancel out or add, depending on $\sgn(\rho, j)$. Consequently we get $c_{\hat \rho}$ which is either $0$, $c_\rho$ or $2c_\rho$ exactly as required. Moreover, if $k=i_s$ then when integrating we get additional factor $\frac{1}{s^I_t+1}$ which is the inverse of the multiplicity of $q_k$ in the component under consideration. This factor is incorporated to $\frac{1}{\hat I!}$ in the formula for $H_{s+1}$. This completes the proof.
\end{proof}

Assume that we are in the region where $q_i>q_j$ if $i<j$. Then, in the 3-dimensional case Theorem \ref{thm_first_int} gives
$$
\begin{aligned}
H_2(q,p)=& \frac{1}{3}\left(p_1^3+p_2^3+p_3^3+3e^{-(q_1-q_2)}(p_1^2p_2+p_1p_2^2)
+3e^{-(q_2-q_3)}(p_2^2p_3+p_2p_3^2)\right.\\
&\left.+3e^{-(q_1-q_3)}(p_1^2p_3+p_1p_3^2)+6e^{-(q_1-q_3)}p_1p_2p_3\right).
\end{aligned}
$$
In dimension 4 we get
$$
\begin{aligned}
H_3(q,p)=&\frac{1}{4}\left(p_1^2+p_2^2+p_3^2+p_4^2+ \sum_{1\leq i<j\leq4}e^{-(q_i-q_j)}(4p_i^3p_j+6p_i^2p_j^2+4p_ip_j^3)\right.\\
&+\sum_{1\leq i<j<k\leq4}\left((8e^{-(q_i-q_k)}+4e^{-(2q_i-q_j-q_k)})p_i^2p_jp_k+(8e^{-(q_i-q_k)}+4e^{-(q_i+q_j-2q_k)})p_ip_jp_k^2\right)\\
&\left.+\sum_{1\leq i<j<k\leq4}12e^{-(q_i-q_k)}p_ip_j^2p_k + (16e^{-(q_1-q_4)}+8e^{-(q_1+q_2-q_3-q_4)})p_1p_2p_3p_4\right).
\end{aligned}
$$
It is illuminating to see how the first integrals behave on the singular set. For instance $H_2=H_0^3$ if $q_1=q_2=q_3$. A similar formula holds for all $H_i$ in all dimensions.

In what follows, for the sake of convenience, we will often consider first integrals $H_i$ defined in Theorem \ref{thm_first_int} rescaled by $i+1$. In particular, in the following section we shall have
$$
H_1(q,p)=g(p,p).
$$

\section{Dissipative prolongation of 2-peakons}
Assume that $n=2$ and consider a geodesic $t\mapsto q(t)$, $t\geq 0$, of the metric $g$ corresponding to a 2-peakon solution to the Camassa-Holm equation. Without loss of generality we assume that
$$
\norm{\dot q(t)}=1,
$$
which means that $H_1=1$ along the geodesic. Moreover, without loss of generality, we consider the region of $\R^2$ where $q_1>q_2$. Then
$$
g=\frac{1}{1-e^{-2s}}\left(
\begin{array}{c c}
1 & -e^{-s}\\
-e^{-s} & 1\\
\end{array}
\right)
$$
where we denote
$$
s=q_1-q_2.
$$

\begin{lemma}\label{lemma1}
Let $q=(q_1,q_2)\in\R^2$. Then
$$
\max_p\{|H_0(q,p)|\ |\ H_1(q,p)=1\}=\frac{\sqrt{2}}{\sqrt{1+e^{-s}}}.
$$
\end{lemma}
\begin{proof}
Recall that $g=h^{-1}$. Let $V=h(p,.)$ be the vector dual to the co-vector $p$ with respect to the metric $g$. Then
$H_0(q,p)=g(X,V)$, where $X=(1,1)$ is the Killing vector field introduced in Proposition \ref{prop1}. Clearly $H_1(q,p)=1$ is equivalent to $g(V,V)=1$ and it follows that the maximum of $H_0$ is attained if $V$ is the unit vector proportional to $X$.
\end{proof}

\begin{lemma}\label{lemma2}
If for a given geodesic $t\mapsto q(t)$ 
$$
|H_0|=|p_1+p_2|>1
$$
holds or
$$
p_1-p_2>0\qquad \mathrm{at}\quad t=0,
$$
then it never approaches the singular set $q_1=q_2$ and consequently there is no collision for the corresponding 2-peakon solution to the Camassa-Holm equation.
\end{lemma}
\begin{proof}
Denote $H_0^{max}(s)=\frac{\sqrt{2}}{\sqrt{1+e^{-s}}}$. Then $H_0^{max}$ is an increasing function of $s$. Moreover
$$
\lim_{s\to 0}H_0^{max}=1.
$$
Thus, if $|H_0|>1$ then there exists $\epsilon>0$ such that $q_1(t)-q_2(t)>\epsilon$ for all $t>0$, because $H_0^{max}\geq |H_0|$ and $H_0$ is constant along geodesics.

Let $X^\perp=(1,-1)$. It is a vector field perpendicular with respect to $g$ to the Killing vector field $X$. Note that $g(\dot q,X^\perp)=p_1-p_2$ and if $g(\dot q(0),X^\perp)>0$  then the geodesic emerges from $q(0)$ in a direction that goes away from the singular set $q_1=q_2$, i.e. $\frac{ds}{dt}|_{t=0}>0$. Moreover, since $H_0^{max}$ is an increasing function of $s$ we get that  $\dot q(t)$ is never parallel to $X$, because $H_0^{max}$ is attained for vectors parallel to $X$. Thus $g(\dot q(t),X^\perp)>0$ for all $t$ and $q_1(t)-q_2(t)$ increases along the geodesic.
\end{proof}

Lemma \ref{lemma2} can be inverted. 
\begin{lemma}\label{lemma3}
If for a given geodesic $t\mapsto q(t)$
$$
|H_0|=|p_1+p_2|< 1
$$
holds and
$$
p_1-p_2<0 \qquad \mathrm{at}\quad t=0,
$$
then it approaches the singular set $q_1=q_2$ in a finite time. If $|H_0|=1$ and $p_1-p_2<0$ at $t=0$ then the geodesic approaches the singular set asymptotically as $t\to\infty$.
\end{lemma}
\begin{proof}
Assume $H_0=c$ along the geodesic. By Proposition \ref{prop2} we have
$$
g(\dot q, X)=H_0, \qquad g(\dot q,\dot q)=1.
$$
One easily solves this algebraic system for $\dot q_1$ and $\dot q_2$ and finds that, under the assumption $p_1-p_2<0$, the function $s(t)=q_1(t)-q_2(t)$ satisfies the following equation
$$
\dot s=-\sqrt{(2-c^2(1+e^{-s}))(1-e^{-s})}.
$$
Assume $|c|<1$. Then
$$
\sqrt{(2-c^2(1+e^{-s}))(1-e^{-s})}\geq\gamma\sqrt{1-e^{-s}}
$$
where $\gamma=\sqrt{2(1-c^2)}$ is constant along geodesics and one easily checks that if
$$
\dot s=-\sqrt{1-e^{-s}}
$$
then $s$ goes to $0$ in a finite time.

Similarly, if $|c|=1$ then
$$
\dot s=-(1-e^{-s})
$$
and one easily checks that $s(t)\to 0$ as $t\to\infty$.
\end{proof}

The condition on the sign of $p_1-p_2$ in Lemma \ref{lemma2} and Lemma \ref{lemma3} is merely a technical assumption that chooses an orientation of a geodesic such that it goes towards or backwards the singular set $q_1=q_2$. Now we can state the following
\begin{theorem}\label{thm1}
A 2-peakon solution to the Camassa-Holm equation with energy $H_1=c>0$ collides if and only if its momentum satisfies $|H_0|<\sqrt{c}$ and $p_1-p_2>0$ at $t=0$. Equivalently $p_1<0$ and $p_2>0$.
\end{theorem}
\begin{proof}
Note that a simple rescaling implies that one can consider the case $c=1$ only.  The first part of the theorem is a compilation of Lemma \ref{lemma2} and \ref{lemma3}. To get the second part of the theorem fix a vector $V=(v_1,v_2)$ and let $p=g(V,.)$. Then $p_1<0$ and $p_2>0$ is equivalent to
$$
v_1-e^{-s}v_2<0,\qquad -e^{-s}v_1+v_2>0.
$$
On the other hand, this is exactly the condition on $V$ that ensures that $g(V,V)=1$ and $|g(V,X)|<1$ with additional assumption $p_1-p_2>0$. Indeed, if $g(V,X)=1$ and $g(V,V)=1$ then $v_1=e^{-s}$ and $v_2=1$ and if $g(V,X)=-1$ and $g(V,V)=1$ then $v_1=-1$ and $v_2=-e^{-s}$.
\end{proof}

Now, we proceed to the main result of this section. Let us consider a geodesic heading to the singular set $q_1=q_2$ such that $H_1=1$ and $|H_0|=c$. According to Lemma \ref{lemma2} and Lemma \ref{lemma3} the geodesic is tangent to one of the two vector fields $V^+_c$ and $V^-_c$ defined by
$$
g(V_c^\pm,V_c^\pm)=1,\quad g(X,V_c^\pm)=\pm c, \quad g(X^\perp, V_c^\pm)<0.
$$
where as before $X=\partial_{q_1}+\partial_{q_2}$ is the Killing vector field and $X^\perp=\partial_{q_1}-\partial_{q_2}$ is a vector field perpendicular to $X$. This gives
$$
\begin{aligned}
&V_c^+=\frac{1}{2}\left(c(1+e^{-s})(\partial_{q_1}+\partial_{q_2})- \sqrt{(2-c^2(1+e^{-s}))(1-e^{-s})}(\partial_{q_1}-\partial_{q_2})\right)\\
&V_c^-=\frac{1}{2}\left(-c(1+e^{-s})(\partial_{q_1}+\partial_{q_2})- \sqrt{(2-c^2(1+e^{-s}))(1-e^{-s})}(\partial_{q_1}-\partial_{q_2})\right).
\end{aligned}
$$
The vector fields are originally defined in the region $q_1>q_2$, but extend uniquely to $\R^2$ such that
$$
V_c^+(q_1,q_2)=V_c^-(q_2,q_1).
$$
Note that the vector fields are not Lipshitz at points of the singular set $q_1=q_2$. Thus the integral curves of $V_c^\pm$ are not unique at these points.
\begin{lemma}\label{lemma4a}
Any integral curve of $V_c^+$ or $V_c^-$ corresponds to a solution of the Camassa-Holm equation.
\end{lemma}
\begin{proof}
An integral curve corresponds to a 2-peakon that collides at certain time $t^*$, then evolves as a 1-peakon for some time (arbitrary long) and finally splits to become a new 2-peakon. If $p_1+p_2=0$ then this corresponds to a creation of a peakon-antipeakon pair. The general proof is included in the proof of Theorem \ref{thm3} in the next section.
\end{proof}

\begin{theorem}\label{thm2}
A dissipative weak 2-peakon solution after a collision time $t^*$ becomes a 1-peakon solution with energy $H_0^2$. The momentum $H_0$ is preserved after the collision. The energy lost equals $H_1-H_0^2$.
\end{theorem}
\begin{proof}
The solution corresponds to an integral curve of $V_c^\pm$ that stays in the singular set once it meets it. Introducing the variable $q$ that parametrises the singular set by the mapping $q\mapsto(q,q)$ we can consider a one-dimensional metric on the singular set with matrix $\tilde g=(1)$. Note that $\tilde g$ coincides with the restriction of $g$ to the singular set. The Killing vector field $X$ for $\tilde g$ is of the form $\partial_q$. Since $V_c^\pm$ restricted to the set $q_1=q_2$ equals $\pm c(\partial_{q_1}+\partial_{q_2})$, which corresponds to $\pm c\partial_q$, we get that the value of the momentum $H_0$ for the 1-peakon equals the original value of $H_0$.

According to \cite{J} it is enough to check that the solution satisfies conditions \eqref{dis1} and \eqref{dis2}. This is straightforward. Namely, the first condition is automatically satisfied, because $\partial_xu$ is estimated by a constant on each interval $[0,t^*)$ and $[t^*,\infty)$ as a solution to \eqref{hameq} with $n=2$ or $n=1$, respectively. The second condition follows from the fact that, due to Theorem \ref{thm1}, $H_0^2\leq H_1$ provided that a 2-peakon solution admits a collision. 
\end{proof}

\section{Dissipative prolongation of $n$-peakons}
We return now to the general case of $n$-peakon solutions to the Camassa-Hom equation. For this we consider a geodesic $t\mapsto q(t)$ of the metric $g=h^{-1}$ where as before
$$
h=(e^{-(q_i-q_j)})_{i,j=1,\ldots,n}
$$
and we assume, without loss of generality, that $q_1>q_2>\cdots>q_n$.

\begin{lemma}\label{lemma4}
Let $t\mapsto q(t)$ be a geodesic of $g$ approaching the singular set $\{q_k=q_{k+1}\}$ at time $t^*$. Then the limit
$$
\lim_{t\to t^*}p_k(t)+p_{k+1}(t)
$$
as well as
$$
\lim_{t\to t^*} p_j(t), \qquad j\neq k,k+1
$$
exist and are finite. Moreover the vector
$$
V=\lim_{t\to t^*}\dot q(t)
$$
exists and is tangent to the singular set. Finally, the $(n-1)$-dimensional restriction $\tilde g$ of the metric $g$ to the singular set is a well defined metric and $\tilde g(V,V)\leq g(\dot q,\dot q)$.
\end{lemma}
\begin{proof}
The proof consists of two ingredients. Firstly we analyse the asymptotic properties of $g$ as $q_k-q_{k+1}\to 0$ and find estimates for coefficients of  $\dot q$. Then we exploit the bi-Hamiltonian structure of the equations and prove the existence of $\lim_{t\to t^*}\dot q(t)$. We denote for simplicity $x_i=e^{q_i}$ and find the following
$$
\begin{aligned}
&g_{11}=\frac{x_1^2}{x_1^2-x_2^2},\\
&g_{ii}=\frac{(x_{i-1}^2-x_{i+1}^2)x_i^2}{(x_{i-1}^2-x_i^2)(x_i^2-x_{i+1}^2)},\qquad i=2,\ldots,n-1,\\
&g_{nn}=\frac{x_{n-1}^2}{x_{n-1}^2-x_n^2},\\
&g_{ii+1}=g_{i+1i}=-\frac{x_ix_{i+1}}{x_i^2-x_{i+1}^2},\qquad i=1,\ldots,n-1,
\end{aligned}
$$
and other coefficients $g_{ij}$ of metric $g$ are zero.  Later, it will be more convenient to consider $g$ in the basis
$$
B=(\partial_{q_1},\ldots,\partial_{q_{k-1}},\partial_{q_k}+\partial_{q_{k+1}},\partial_{q_k}-\partial_{q_{k+1}},\partial_{q_{k+2}},\ldots,\partial_{q_n} ).
$$
For this we compute directly 
$$
\begin{aligned}
&g(\partial_{q_k}+\partial_{q_{k+1}},\partial_{q_k}+\partial_{q_{k+1}})=\frac{2(x_{k-1}^2x_{k+1}-x_kx_{k+2}^2)x_kx_{k+1} -(x_k^2x_{k+1}^2+x_{k-1}^2x_{k+2}^2)(x_k-x_{k+1})}{(x_{k-1}^2-x_k^2)(x_k+x_{k+1})(x_{k+1}^2-x_{k+2}^2)}\\
&g(\partial_{q_k}+\partial_{q_{k+1}},\partial_{q_k}-\partial_{q_{k+1}})=\frac{x_k^2x_{k+1}^2-x_{k-1}^2x_{k+2}^2}{(x_{k-1}^2-x_k^2)(x_{k+1}^2-x_{k+2}^2)}\\
&g(\partial_{q_k}-\partial_{q_{k+1}},\partial_{q_k}-\partial_{q_{k+1}})=\frac{2(x_{k-1}^2x_{k+1}+x_kx_{k+2}^2)x_kx_{k+1} -(x_k^2x_{k+1}^2+x_{k-1}^2x_{k+2}^2)(x_k+x_{k+1})}{(x_{k-1}^2-x_k^2)(x_k-x_{k+1})(x_{k+1}^2-x_{k+2}^2)}
\end{aligned}
$$
In particular, $g(\partial_k+\partial_{k+1},\partial_k+\partial_{k+1})$ converges to
$$
\frac{(x_{k-1}^2-x_{k+2}^2)x_k^2}{(x_{k-1}^2-x_k^2)(x_k^2-x_{k+2}^2)}
$$
as $q_k-q_{k+1}\to 0$. Let us parametrize the singular set by $(q_1,q_2,\ldots,q_{n-1})$ as follows
$$
(q_1,q_2,\ldots,q_{n-1})\mapsto(q_1,q_2,\ldots,q_{k-1},q_k,q_k,q_{k+1},\ldots ,q_{n-1}).
$$
We get that $g$ restricted to the singular set is well defined. Moreover it is of the form of the original $g$ but in one dimension less. Equivalently one can see that $h$ restricted to the singular set is non-degenerate.

Consider a geodesic $t\mapsto q(t)$ that attains the set $\{q_k=q_{k+1}\}$ at time $t^*$. Without loss of generality we assume that $\norm{\dot q(t)}=1$ for $t<t^*$. We shall denote
$$
s=q_k-q_{k+1}.
$$
Thus
$$
\lim_{t\to t^*}s(t)=0.
$$
Let us write $\dot q(t)=\sum_{i=1}^na_i(t)\partial_{q_i}$ and denote 
$$
b(t)=\frac{1}{2}(a_k(t)+a_{k+1}(t)),\qquad c(t)=\frac{1}{2}(a_k(t)-a_{k+1}(t)).
$$
Note that $g(\partial_{q_k}-\partial_{q_{k+1}},\partial_{q_k}-\partial_{q_{k+1}})$ computed above is the only coefficient of $g$ in the basis $B$ that is unbounded as $q_k-q_{k+1}\to 0$. Hence
$$
\lim_{t\to t^*}c(t)=0,
$$
because $\norm{\dot q}$ is constant. Precisely, $c(t)$ has the following asymptotic behaviour
$$
c(t)=O(\sqrt{s(t)}),
$$
because $g(\partial_{q_k}-\partial_{q_{k+1}},\partial_{q_k}-\partial_{q_{k+1}})^{-1}=O(s)$. Additionally a simple comparison of $g$ with the standard flat metric in $\R^n$ gives that all coefficients $a_i(t)$ are bounded. 

Ultimately we shall prove that the limits $\lim_{t\to t^*}a_i(t)$, $i\neq k, k+1$, and $\lim_{t\to t^*}b(t)$ exist. However, we consider first the associated co-vector $p(t)=g(\dot q(t),.)$ with entries
$$
p_i=a_{i-1}g_{i-1i}+a_ig_{ii}+a_{i+1}g_{ii+1}
$$
(accordingly modified for $i=1$ and $i=n$). Note that, since all $a_i$ are bounded, we get that only $p_k$ and $p_{k+1}$ are possibly unbounded as $t\to t^*$. However
$$
\begin{aligned}
p_k+p_{k+1}&=-a_{k-1}\frac{x_{k-1}x_k}{x_{k-1}^2-x_k^2}-a_{k+2}\frac{x_{k+1}x_{k+2}}{x_{k+1}^2-x_{k+2}^2}\\
&+a_k\frac{x_k(x_{k-1}^2+x_kx_{k+1})}{(x_k+x_{k+1})(x_{k-1}^2-x_k^2)} +a_{k+1}\frac{x_{k+1}(x_kx_{k+1}+x_{k+2}^2)}{(x_k+x_{k+1})(x_{k+1}^2-x_{k+2}^2)}
\end{aligned}
$$
is bounded, and for
$$
\begin{aligned}
p_k-p_{k+1}&=-a_{k-1}\frac{x_{k-1}x_k}{x_{k-1}^2-x_k^2}+a_{k+2}\frac{x_{k+1}x_{k+2}}{x_{k+1}^2-x_{k+2}^2}\\
&+a_k\frac{x_k(x_{k-1}^2-x_kx_{k+1})}{(x_k-x_{k+1})(x_{k-1}^2-x_k^2)} -a_{k+1}\frac{x_{k+1}(x_kx_{k+1}-x_{k+2}^2)}{(x_k-x_{k+1})(x_{k+1}^2-x_{k+2}^2)}
\end{aligned}
$$
we have
$$
\lim_{t\to t^*}|p_k(t)-p_{k+1}(t)|=\infty.
$$
Nevertheless, since $x_k(t)\to x_{k+1}(t)$ as $t\to t^*$, we approximately have
$$
s(t)(p_k(t)-p_{k+1}(t))^2=s(t)c(t)^2\frac{(x_k+x_{k+1})^2}{(x_k-x_{k+1})^2}+O(s(t))
$$
which is bounded, because $c(t)=O(\sqrt{s(t)})$.

Now we will prove that all $p_i(t)$, $i\neq k, k+1$, as well as $p_k(t)+p_{k+1}(t)$ and $s(t)(p_k(t)-p_{k+1}(t))^2$ have limits as $t\to t^*$. Once we do this we will also have that $V=\lim_{t\to t^*}\dot q(t)$ exists and is tangent to the singular set, because we already know that $\lim_{t\to t^*}c(t)=0$. For this we will prove that all first integrals $H_i$ are polynomial functions of $\psi_k(t)$, $\xi_k(t)$ and $p_i(t)$ for $i\neq k, k+1$, where
$$
\psi_k(t)=p_k(t)+p_{k+1}(t)
$$
and
$$
\xi_k(t)=\sqrt{s(t)}(p_k(t)-p_{k+1}(t))
$$
with bounded and continuous coefficients on the interval $[0,t^*]$. These will imply that the limits $p_i(t)$, $i\neq k, k+1$, as well as $p_k(t)+p_{k+1}(t)$ and $s(t)(p_k(t)-p_{k+1}(t))^2$ with $t\to t^*$ exist because solutions to the polynomial equations $H_i=\const$ depend continuously on the coefficients of polynomials \cite[Appendix V, Section 4]{W}. The continuity of roots is usually deduced from Rouch\'e's theorem in a context of a complex polynomial. In our setting we have  a curve $t\mapsto p(t)$ which give rise to roots $(p_i,\psi_k,\xi_k)$,  $i\neq k, k+1$,  of our system for each $t\in[0,t^*)$. Since the roots are real (and bounded) we get in a limit also a real root, even though, \emph{a priori}, we apply the continuity result to the complexified equations.

The momentum $H_0$ is clearly a polynomial function of $p_i$, $i\neq k, k+1$, and $\psi_k$ with constant coefficients. Further, one can also easily write the energy $H_1$ as a polynomial function of $p_i$, $i\neq k, k+1$, $\psi_k$ and $\xi_k$. Indeed 
$$
\begin{aligned}
H_1(q,p)&=\frac{1}{2}(1+e^{-s})\psi_k^2+\frac{1}{2}(1-e^{-s})\frac{1}{s}\xi_k^2\\
&+\sum_{i\neq k,k+1}\left((e^{-|q_k-q_i|}+e^{-|q_{k+1}-q_i|})\psi_kp_i+(e^{-|q_k-q_i|}-e^{-|q_{k+1}-q_i|})\frac{1}{\sqrt{s}}\xi_kp_i\right)\\
&+\sum_{i,j\neq k,k+1}e^{-|q_i-q_j|}p_ip_j.
\end{aligned}
$$
Then one sees that the coefficients of the polynomial are continuous functions of $q_i$'s that are bounded on the trajectory $t\mapsto q(t)$. Namely 
$$
\lim_{t\to t^*}\frac{1-e^{-s(t)}}{s(t)}=1, \qquad\lim_{t\to t^*}\frac{e^{-|q_k(t)-q_i(t)|}-e^{-|q_{k+1}(t)-q_i(t)|}}{\sqrt{s(t)}}=0.
$$
To prove that $H_i$ are of the desired form, for $i>1$ it is sufficient to analyse the recursive operator $S$ used in \eqref{firstint} and given explicitly in terms of $P_0$ and $P_1$ in Theorem \ref{thmbiham}. In fact it is sufficient to consider $P_0$ and show that when it acts on $dH_i$ then the terms that give new factors of the form $(p_k-p_{k+1})$ in $dH_{i+1}$ are always accompanied with a coefficient $A$ such that $\frac{A}{\sqrt{s}}$ is bounded. This can be checked directly in a way similar to the decomposition of $H_1$ presented above. For instance, we rewrite the sum $\sum_{i,j=1}^np_ie^{-|q_i-q_j|}\partial_{p_i}\wedge\partial_{q_j}$ in the following way
$$
\begin{aligned}
\sum_{i,j=1}^np_ie^{-|q_i-q_j|}\partial_{p_i}\wedge\partial_{q_j}&=\frac{1}{2}(1-e^{-s})(p_k\partial_{p_k}-p_{k+1}\partial_{p_{k+1}})\wedge(\partial_{q_k}-\partial_{q_{k+1}})\\
&+\frac{1}{2}\sum_{i\neq k,k+1}\left((e^{-|q_k(t)-q_i(t)|}+e^{-|q_{k+1}(t)-q_i(t)|})(p_k\partial_{p_k}+p_{k+1}\partial_{p_{k+1}})\right.\\
&\qquad\qquad+\left.(e^{-|q_k(t)-q_i(t)|}-e^{-|q_{k+1}(t)-q_i(t)|})(p_k\partial_{p_k}-p_{k+1}\partial_{p_{k+1}})\right)\wedge\partial_{q_i}\\
&+\frac{1}{2}\sum_{i\neq k,k+1}p_i\partial_{p_i}\wedge\left((e^{-|q_k(t)-q_i(t)|}+e^{-|q_{k+1}(t)-q_i(t)|})(\partial_{q_k}+\partial_{q_{k+1}})\right.\\
&\qquad\qquad+\left.(e^{-|q_k(t)-q_i(t)|}-e^{-|q_{k+1}(t)-q_i(t)|})(\partial_{q_k}-\partial_{q_{k+1}})\right)\\
&+\sum_{i,j\neq k,k+1}p_ie^{-|q_i-q_j|}\partial_{p_i}\wedge\partial_{q_j}.
\end{aligned}
$$
and similar formulae hold for other components of $P_0$. Note that new factor $(p_k-p_{k+1})$ can appear when $P_1^{-1}$ is applied to $\partial_{q_k}-\partial_{q_{k+1}}$ or when $p_k\partial_{p_k}-p_{k+1}\partial_{p_{k+1}}$ acts on $\psi_k$ by differentiation. In both cases there is a coefficient in $P_0$ of order $O(\sqrt{s})$ at least like in the case of $H_1$.

Now, in order to complete the proof it is sufficient to argue that $\tilde g(V,V)\leq H_1$. This follows from the fact that, according to our earlier analysis of $g$ one can decompose $g(\dot q(t),\dot q(t))$ as a sum $c(t)^2g_1(t)+c(t)g_2(t)+g_3(t)$ where $g_2$ is bounded, $g_3$ converges to $\tilde g(V,V)$ and $g_1$ is positive and such that $c(t)^2g_1(t)$ is bounded.
\end{proof}

\begin{theorem}\label{thm3}
A dissipative weak $n$-peakon solution after a collision time $t^*$ becomes an $(n-1)$-peakon solution such that the momentum $H_0$ is preserved after the collision.
\end{theorem}
\begin{proof}
Let $u(t,x)$ be a general $n$-peakon solution. We compute
$$
\begin{aligned}
&u_t=\sum_i \dot p_ie^{-|x-q_i|}+\sum_ip_i\dot q_ie^{-|x-q_i|}\sgn(x-q_i),\\
&u_x=-\sum_i p_ie^{-|x-q_i|}\sgn(x-q_i),\\
&u_{xx}=\sum_ip_ie^{-|x-q_i|}-2\sum_ip_i\delta(x-q_i),\\
&u_{xxt}=\sum_i\dot p_ie^{-|x-q_i|}+\sum_ip_i\dot q_ie^{-|x-q_i|}\sgn(x-q_i) -2\sum_i\dot p_i\delta(x-q_i)+2\sum_ip_i\dot q_i\delta'(x-q_i),\\
&u_{xxx}=-\sum_i p_ie^{-|x-q_i|}\sgn(x-q_i)-2\sum_ip_i\delta'(x-q_i),
\end{aligned}
$$
Where $\delta$ is the Dirac delta and $\delta'$ is its derivative understood in the functional sense. Substituting the formulae to \eqref{eqCH} we get
$$
\sum_{i,j}(\dot p_i-2p_ip_je^{-|x-q_i|}\sgn(x-q_j))\delta(x-q_i)- \sum_{i,j}p_i(\dot q_i-p_je^{-|x-q_i|})\delta'(x-q_i)=0.
$$
Evaluating this on a test function $\varphi$ and having in mind that integration by parts gives $f\delta[\varphi]=-f'\delta[\varphi]-f\delta[\varphi']$ we get the following
$$
\sum_{i,j}(\dot p_i-p_ip_je^{-|x-q_i|}\sgn(x-q_j))\delta(x-q_i)[\varphi] +\sum_{i,j}p_i(\dot q_i-p_je^{-|x-q_i|})\delta(x-q_i)[\varphi']=0
$$
which holds for any $\varphi$. Thus, if all $q_i$ are different we get that $(q_i,p_i)$ are differentiable and satisfy the system
$$
\dot p_i-\sum_j p_ip_je^{-|q_j-q_i|}\sgn(q_i-q_j)=0,\qquad \dot q_i-\sum_jp_je^{-|q_j-q_i|}=0,
$$
which coincides with the Hamiltonian system \eqref{hameq}. On the other hand, if $q_k=q_{k+1}$ for some $k$ then $p_k+p_{k+1}$ is differentiable and satisfies
\begin{equation}\label{hamredeq}
\frac{d}{dt}(p_k+p_{k+1})-\sum_j(p_k+p_{k+1})p_je^{-|q_j-q_k|}\sgn(q_k-q_j)=0
\end{equation}
and the additional equations are unchanged (note however that the terms involving $p_k$ and $p_{k+1}$ can be grouped and written in terms of $\psi_k=p_k+p_{k+1}$). Conversely any solution $(q,p)$ of \eqref{hameq} such that $p_k+p_{k+1}$ is a well defined differentiable function satisfying \eqref{hamredeq} at points where $q_k=q_{k+1}$ give rise to a week solution of the Camassa-Holm equation.

Now, according to Lemma \ref{lemma4} the vector
$$
(p_1(t^*),\ldots,p_{k-1}(t^*),p_k(t^*)+p_{k+1}(t^*),p_{k+2}(t^*),\ldots,p_n(t^*))
$$
is well defined at the collision time $t^*$ and it defines the initial data for $(n-1)$-peakon. Note that that equation \eqref{hamredeq} together with the remaining equations for $p_i$'s and $q_i$'s (excluding the equations for $p_k$ and $p_{k+1}$) has exactly the form of \eqref{hameq} in $n-1$ dimensions. Thus an $n$-peakon after a collision can evolve as $(n-1)$-peakon. The momentum $H_0$ is preserved after the collision. Moreover, Lemma \ref{lemma4} implies that the energy $H_1$ decreases for such a solution thus \eqref{dis2} is satisfied. The condition \eqref{dis1} is satisfied too, since, as in the 2-dimensional case, $\partial_xu$ is estimated by a constant on $[0,t^*)$ and on $[t^*,\infty)$. 
\end{proof}

\paragraph{Splitting of $n$-peakons.} It follows from the proof of Theorem \ref{thm3} that an $n$-peakon can split at any time to become an $n+1$-peakon. The momentum is preserved in this process however the energy grows. Moreover, the solutions after a splitting violate condition \eqref{dis1}. Indeed, we only have an estimate of $\partial_xu$ on $[t,\infty)$ for each $t>t^*$, where $t^*$ is the splitting time and in fact $\partial_xu$ is arbitrary large for $t$ close to $t^*$.

\end{document}